\documentclass{amsart}

\usepackage{tikz}
\usepackage{amsmath}
\usepackage{amsthm}
\usepackage{amssymb}
\usepackage{mathtools}

\newcommand{\Q}{\mathbb{Q}}

\newcommand{\R}{\mathbb{R}}
\newcommand{\N}{\mathbb{N}}

\newcommand{\dom}{\operatorname{dom}}

\newcommand{\supp}{\operatorname{supp}}

\newcommand{\B}{\mathcal{B}}

\newcommand{\A}{\mathcal{A}}

\newcommand{\M}{\mathcal{M}}

\newcommand{\one}{\mathbf{1}}

\newcommand{\ol}{\overline}

\newtheorem{theorem}{Theorem}[section]
\newtheorem{lemma}[theorem]{Lemma}

\newtheorem{corollary}[theorem]{Corollary}

\newtheorem{proposition}[theorem]{Proposition}

\theoremstyle{definition}
\newtheorem{definition}[theorem]{Definition}
\newtheorem{question}[theorem]{Question}

\numberwithin{equation}{section}

\title{Effective weak and vague convergence of measures on the real line}
\author{Diego A. Rojas}
\address{Department of Mathematics\\
Iowa State University\\
Ames, Iowa 50011, United States}
\email{darojas@iastate.edu}

\begin{document}

\begin{abstract} 
    We expand our effective framework for weak convergence of measures on the real line by showing that effective convergence in the Prokhorov metric is equivalent to effective weak convergence. 
    In addition, we establish a framework for the study of the effective theory of vague convergence of measures. 
    We introduce a uniform notion and a non-uniform notion of vague convergence, and we show that both these notions are equivalent. 
    However, limits under effective vague convergence may not be computable even when they are finite. 
    We give an example of a finite incomputable effective vague limit measure, and we provide a necessary and sufficient condition so that effective vague convergence produces a computable limit. 
    Finally, we determine a sufficient condition for which effective weak and vague convergence of measures coincide. 
    As a corollary, we obtain an effective version of the equivalence between classical weak and vague convergence of sequences of probability measures.

\end{abstract}

\maketitle

\vspace{-0.3in}

\section{Introduction}

Recently, McNicholl and D. Rojas \cite{MR21} developed a framework to study the effective theory of weak convergence of measures on $\R$. 
Recall that a sequence of finite Borel measures $\{\mu_n\}_{n\in\N}$ on a separable metric space $X$ \emph{weakly converges} to a measure $\mu$ if, for every bounded continuous real-valued function $f$ on $X$, $\lim_n\int_Xfd\mu_n=\int_Xfd\mu$. 
Weak convergence is used extensively in probability theory, particularly in the study of optimization problems in stochastic dynamic programming \cite{B89} and controlled Markov processes \cite{BB96}. 
This paper partially serves as a continuation of \cite{MR21}. 

As seen in \cite{HR09b}, the space $\M(X)$ of finite Borel measures on a computable metric space $X$ forms a computable metric space when equipped with the Prokhorov metric. 
It is well known that the Prokhorov metric, introduced by Prokhorov \cite{P56} in 1956, metrizes the topology of weak convergence of measures. 
Thus, if a uniformly computable sequence in $\M(X)$ converges effectively in the Prokhorov metric, then the limit is a computable measure. 
This leads us to the following.

\begin{question}\label{qn:equiv}
    Are effective weak convergence and effective convergence in the Prokhorov metric equivalent?
\end{question}

In this paper, we provide a positive answer to this question in the case of $\M(\R)$. 
In the classical theory, the equivalence between weak convergence and convergence in the Prokhorov metric uses the Portmanteau Theorem, a characterization theorem for weak convergence originally due to Alexandroff \cite{A41}. 
Thus, the key to answering Question \ref{qn:equiv} is the effective Portmanteau Theorem (Theorem 5.1 in \cite{MR21}). 

Although not as commonly studied, another notion of convergence for sequences of measures in analysis and probability theory is \emph{vague convergence}. 
A sequence of (not necessarily finite) Borel measures $\{\mu_n\}_{n\in\N}$ on a separable metric space $X$ \emph{vaguely converges} to a measure $\mu$ if, for every compactly-supported continuous real-valued function $f$ on $X$, $\lim_n\int_Xfd\mu_n=\int_Xfd\mu$. 
If $\{\mu_n\}_{n\in\N}$ is a sequence of probability measures, then $\{\mu_n\}_{n\in\N}$ converges vaguely if and only if it converges weakly (see \cite{C01}). 
This fact is used in \cite{MTY09} to define an effective notion of convergence for probability measures on $\M(\R)$. 
However, this definition is not a suitable effective analogue to classical vague convergence for non-probability measures on $\M(\R)$.
We are thus led to the following.

\begin{question}
    What is a suitable definition of effective vague convergence?
\end{question}

Following \cite{MR21}, we propose two answers to this question, one of which is uniform (Definition \ref{def:uev}) and one of which is not (Definition \ref{def:ev}). 
As in the case of effective weak convergence, we show that these definitions are equivalent (Theorem \ref{thm:evc.equiv}). 
In contrast to effective weak convergence, we also show that effective vague convergence does not guarantee a computable limit even when the limit is finite.
Nevertheless, it is possible to obtain a computable limit under effective vague convergence.

As previously stated, classical weak and vague convergence coincide at sequences of probability measures. 
Along the same vein, we determine the point at which effective weak and vague convergence coincide (Theorem \ref{thm:evc2ewc}).
This yields an effective version of the correspondence between classical weak and vague convergence for probability measures (Corollary \ref{cor:evc2ewc.prob}). 

This paper is divided as follows. 
Section 2 consists of the necessary background in both classical and computable analysis and measure theory. 
Section 3 covers some preliminary material used in latter sections to state and prove the main results of this paper. 
In Section 4, we prove the equivalence between effective weak convergence and effective convergence in the Prokhorov metric in $\M(\R)$. 
In Section 5, we define effective notions of vague convergence in $\M(\R)$ and analyze the aforementioned consequences. 
We conclude in Section 6 with a discussion of the results and implications for future research in this direction.

\section{Background}

\subsection{Background from classical analysis}

In this paper, we denote the set of all continuous functions on $\R$ by $C(\R)$, the set of all bounded continuous functions on $\R$ by $C_b(\R)$, and the set of all compactly-supported continuous functions on $\R$ by $C_K(\R)$. 
We define the \emph{support} of a function $f\in C(\R)$ to be the set $\supp{f}=\ol{\R\setminus f^{-1}(\{0\})}$. 

For $x\in\R$ and $A\subseteq\R$, let $d(x,A):=\inf_{a\in A}|x-a|$. 
For $\epsilon>0$, we denote the open $\epsilon$-neighborhood of $a\in \R$ by $B(a,\epsilon)$.
For $\epsilon>0$ and $A\subseteq\R$, $B(A,\epsilon):=\bigcup_{a\in A}B(a,\epsilon)$ is called the open $\epsilon$-neighborhood of $A$.
We denote the Borel $\sigma$-algebra of $\R$ by $\B(\R)$. 
The Prokhorov metric $\rho$ is defined as follows: for any $\mu,\nu\in\M(\R)$,
\[
    \rho(\mu,\nu)=\inf\{\epsilon>0:(\forall A\in B(\R))[\mu(A)\leq\nu(B(A,\epsilon))+\epsilon\wedge\mu(A)\leq\nu(B(A,\epsilon))+\epsilon]\}.
\] 

\subsection{Background from computable analysis and computable measure theory}

We say that a bounded interval $I\subseteq\R$ is \emph{rational} if each of its endpoints is rational. Fix an effective enumeration $\{I_i\}_{i\in\N}$ of the set of all rational open intervals. 

An open set $U\subseteq\R$ is $\Sigma^0_1$ if $\{i\in\N:I_i\subseteq U\}$ is c.e.. 
Similarly, a closed set $C\subseteq\R$ is $\Pi^0_1$ if $\{i\in\N:I_i\cap C=\emptyset\}$ is c.e..
We denote the set of $\Sigma^0_1$ subsets of $\R$ by $\Sigma^0_1(\R)$, and we denote the set of $\Pi^0_1$ subsets of $\R$ by $\Pi^0_1(\R)$.
We say that $e \in \N$ \emph{indexes} $U \in \Sigma_1^0(\R)$ if $e$ indexes $\{i \in \N\ :\ I_i \subseteq U\}$.  
Indices of sets in $\Pi^0_1(\R)$ are defined analogously.  
A pair $(U,V)$ of $\Sigma_1^0$ sets is indexed by an $e \in \N$ if $U$ is indexed by $(e)_0$ and $V$ is indexed by $(e)_1$.

For a compact set $K\subseteq\R$, a \emph{(minimal cover) name} for $K$ is an enumeration of all minimal finite covers of $K$. We say $K$ is \emph{computably compact} if it has a computable name. An index of $K$ is defined to be an index of a name of $K$.

If $I$ is a rational compact interval, let $P_{\Q}(I)$ denote the space of polygonal functions on $I$ with rational vertices; we will refer to these functions as \emph{rational polygonal functions} on $I$.
When $p\in P_{\Q}[a,b]$,  we may extend $p$ to all of $\R$ by letting $p(x)=p(a)$ for $x<a$ and $p(x)=p(b)$ for $x>b$. 
We may also extend $p$ to be supported on a computably compact set.

Fix a real number $x$.  A \emph{(Cauchy) name} of $x$ is a sequence $\{q_n\}_{n \in \N}$ of rational numbers so that 
$\lim_n q_n = x$ and so that $|q_n - q_{n+1}| < 2^{-n}$ for all $n \in \N$.  

When $I\subseteq\R$ is a compact rational interval and $J\subseteq\R$ is a rational open interval, we let $N_{I,J}=\{f\in C(\R):f[I]\subseteq J\}$. 
A \emph{(compact-open) name} of a function $f\in C(\R)$ is an enumeration of $\{N_{I,J}:f\in N_{I,J}\}$.  
If $f \in C(\R)$, then $f$ is computable if and only $f$ has a computable name.

Fix $x \in \R$. We say $x$ is \emph{computable} if it has a computable name.  An index of such a name is also said to be an index of $f$. We say $x$ is \emph{left-c.e.} (\emph{right-c.e.}) if its left (right) Dedekind cut is c.e..  It follows that $x$ is computable if and only if it is left-c.e. and right-c.e.. A sequence $\{x_n\}_{n \in \N}$ is computable if $x_n$ is computable uniformly in $n$.

A function $f:\subseteq\R\rightarrow\R$ is \emph{computable} if there is a Turing functional $F$ so that
$F(\rho)$ is a name of $f(x)$ whenever $\rho$ is a name of $x$.  An index of such a functional $F$ is also said to be an index of $f$.  We denote the set of all bounded computable functions on $\R$ by $C_b^c(\R)$. We denote the set of all compactly-supported computable functions on $\R$ by $C_K^c(\R)$.

A function $f: \subseteq \R \rightarrow \R$ is \emph{lower semi-computable} if there is a Turing functional $F$ so that 
$F(\rho)$ enumerates the left Dedekind cut of $f(x)$ whenever $\rho$ is a name of $x$.
A function $f: \subseteq \R \rightarrow \R$ is \emph{upper semi-computable} if $-f$ is lower semi-computable.

A function $F:\subseteq C(\R)\rightarrow\R$ is \emph{computable} if there is a Turing functional $\Phi$ 
so that $\Phi(\rho)$ is a name of $F(f)$ whenever $\rho$ is a name of $f$.  
An index of such a functional $\Phi$ is also said to be an index of $F$.

Each of the names we have just discussed can be represented as a point in $\Sigma^\omega$ for a sufficiently large alphabet $\Sigma$.

Suppose $\{a_n\}_{n \in \N}$ is a convergent sequence of reals, and let $a = \lim_n a_n$.  A \emph{modulus of convergence} of $\{a_n\}_{n \in \N}$ is a function $g : \N \rightarrow \N$ so that $|a_n - a| < 2^{-k}$ whenever $k \geq g(n)$.  

A measure $\mu\in\M(\R)$ is \emph{computable} if $\mu(\R)$ is a computable real and $\mu(U)$ is left-c.e. uniformly in an index of $U\in\Sigma^0_1(\R)$;  i.e. it is possible to compute an index of the left Dedekind cut of $\mu(U)$ from an index of $U$.  A sequence of measures $\{\mu_n\}_{n\in\N}$ in $\M(\R)$ is \emph{uniformly computable} if $\mu_n$ is a computable measure uniformly in $n$.

Suppose $\mu \in \M(\R)$ is computable.  A pair $(U,V)$ of $\Sigma^0_1$ subsets of $\R$ is \emph{$\mu$-almost decidable} if $U\cap V=\emptyset$, $\mu(U\cup V)=\mu(\R)$, and $\ol{U\cup V}=\R$.  
If, in addition, $U \subseteq A \subseteq \R - V$, then we say that $(U,V)$ is a $\mu$-almost decidable pair of $A$.  
We then say $A$ is \emph{$\mu$-almost decidable} if it has a $\mu$-almost decidable pair. 
Suppose $(U,V)$ is a $\mu$-almost decidable pair of $A$. Then, $e$ indexes $A$ if $e=\left<i,j\right>$ for some index $i$ of $U$ and some index $j$ of $V$.  We note that $\mu$-almost decidable sets are effective analogues of $\mu$-continuity sets.

\section{Preliminaries}

The following effective notions of weak convergence of measures appear in \cite{MR21}.

\begin{definition}\label{def:ew}
    We say $\{\mu_n\}_{n\in\N}$ \emph{effectively weakly converges} to $\mu$ if for every $f \in C_b^c(\R)$, 
    $\lim_n \int_\R f\ d\mu_n = \int_\R f\ d\mu$ and it is possible to compute an index of a modulus of convergence of 
    $\{\int_\R f\ d\mu_n\}_{n \in \N}$ from an index of $f$ and a bound $B \in \N$ on $|f|$. 
\end{definition}
    
\begin{definition}\label{def:uew}
    We say $\{\mu_n\}_{n\in\N}$ \emph{uniformly effectively weakly converges} to $\mu$ if it weakly converges to $\mu$ 
    and there is a uniform procedure that for any $f \in C_b(\R)$ computes a modulus of convergence of 
    $\{\int_\R f\ d\mu_n\}_{n \in \N}$ from a c.o.-name of $f$ and a bound $B \in \N$ on $|f|$.
\end{definition}

Observe that Definition \ref{def:uew} is uniform, whereas Definition \ref{def:ew} is not uniform. Nevertheless, we have the following.

\begin{theorem}[\cite{MR21}]\label{thm:ewc.equiv}
    Suppose $\{\mu_n\}_{n\in\N}$ is uniformly computable. The following are equivalent.
    \begin{enumerate}
        \item $\{\mu_n\}_{n\in\N}$ is effectively weakly convergent.
        \item $\{\mu_n\}_{n\in\N}$ is uniformly effectively weakly convergent.
    \end{enumerate}
\end{theorem}

As we will see in Section 5, we will model the effective definitions of vague convergence in $\M(\R)$ after Definitions \ref{def:ew} and \ref{def:uew}.

In addition, we need the following pair of definitions.

\begin{definition}\label{def:wit.lwr.upr}
    Suppose $\{a_n\}_{n \in \N}$ is a sequence of reals, and let $g : \subseteq \Q \rightarrow \N$.
    \begin{enumerate}
        \item We say $g$ \emph{witnesses that $\liminf_n a_n$ is not smaller than $a$} if $\dom(g)$ is the left Dedekind cut of $a$ and if $r < a_n$ whenever $r \in \dom(g)$ and $n \geq g(r)$.
        
        \item We say $g$ \emph{witnesses that $\limsup_n a_n$ is not larger than $a$} if $\dom(g)$ is the right Dedekind cut of $a$ and if $r > a_n$ whenever $r \in \dom(g)$ and $n \geq g(r)$.
    \end{enumerate}
\end{definition}

In order to prove our first main result in the next section, we will need the following theorem.

\begin{theorem}[Effective Portmanteau Theorem, \cite{MR21}]\label{thm:ept}
    Let $\{\mu_n\}_{n\in\N}$ be a uniformly computable sequence in $\M(\R)$. The following are equivalent.
    \begin{enumerate}
        \item $\{\mu_n\}_{n\in\N}$ effectively weakly converges to $\mu$.\label{thm:ept::ewc}
            
        \item From $e,B \in \N$ so that $e$ indexes a uniformly continuous $f \in C_b(\R)$ with $|f| \leq B$, it is possible to compute a modulus of convergence of $\{\int_\R f\ d\mu_n\}_{n \in \N}$.\label{thm:ept::uc}
        
        \item $\mu$ is computable, and from an index of $C \in \Pi^0_1(\R)$ it is possible to compute an index of a witness that $\limsup_n \mu_n(C)$ is not larger than $\mu(C)$.\label{thm:ept::clsd}
            
        \item $\mu$ is computable, and from an index of $U \in \Sigma^0_1(\R)$ it is possible to compute an index of a witness that $\liminf_n \mu_n(U)$ is not smaller than $\mu(U)$.\label{thm:ept::opn}
            
        \item $\mu$ is computable, and for every $\mu$-almost decidable $A$, $\lim_n \mu_n(A) = \mu(A)$ and an index of a modulus of convergence of $\{\mu_n(A)\}_{n \in \N}$ can be computed from a $\mu$-almost decidable index of $A$.\label{thm:ept::a.d.}
    \end{enumerate}
\end{theorem}

\section{Effective Convergence in the Prokhorov Metric}

We say that a sequence $\{\mu_n\}_{n\in\N}$ in $\M(\R)$ \emph{converges effectively in the Prokhorov metric} $\rho$ to a measure $\mu$ if there is a computable function $\epsilon:\N\rightarrow\N$ such that for every $n,N\in\N$, $n\geq\epsilon(N)$ implies $\rho(\mu_n,\mu)<2^{-N}$. 
Since $\M(\R)$ forms a computable metric space under $\rho$, it follows that every uniformly computable sequence of measures in $\M(\R)$ converges to a computable measure in $\rho$. 
As $\rho$ metrizes the topology of weak convergence of measures in $\M(\R)$, it is natural to characterize effective convergence in $\rho$ as an effective notion of weak convergence. 
However, it is not immediately clear that effective convergence in $\rho$ can be obtained from effective weak convergence and vice versa.
The following result states that both of these convergence notions coincide for uniformly computable sequences on measures in $\M(\R)$.

\begin{theorem}\label{thm:cpm}
    Suppose $\{\mu_n\}_{n\in\N}$ is uniformly computable. The following are equivalent:
    \begin{enumerate}
        \item $\{\mu_n\}_{n\in\N}$ is effectively weakly convergent;
        \item $\{\mu_n\}_{n\in\N}$ converges effectively in $\rho$.
    \end{enumerate}
\end{theorem}

First, we need the following lemma.

\begin{lemma}\label{lem:eoc}
    Let $\mu\in\M(\R)$ be computable, and let $s>0$ be rational. It is possible to compute an open cover of $\R$ consisting of open balls with radius less that $s$, each of which is a $\mu$-almost decidable set.
\end{lemma}

\begin{proof}
    Adapt the proof of Lemma 5.1.1 in \cite{HR09b} by replacing $\R^+$ with $(0,s)$.
\end{proof}

The proof of the classical version of Theorem \ref{thm:cpm} makes use of the classical Portmanteau Theorem as well as a classical version of Lemma \ref{lem:eoc}. 
As we shall see below, Theorem \ref{thm:cpm} makes use of the effective Portmanteau Theorem as well as Lemma \ref{lem:eoc}. 
However, before proving Theorem \ref{thm:cpm}, we also require the following lemma. 

\begin{lemma}\label{lem:cn}
    If $C\in\Pi^0_1(\R)$, then $\ol{B(C,s)}\in\Pi^0_1(\R)$ for any rational $s>0$.
\end{lemma}

\begin{proof}
    Let $\{I_i\}_{i\in\N}$ be an enumeration of all rational open intervals of $\R$. 
    Then, for each $i\in\N$, $I_i=B(a_i,r_i)$ for some $a_i,r_i\in\Q$ with $r_i>0$. 
    Let $A=\ol{B(C,s)}$, and let $E_A=\{i\in\N:I_i\cap A=\emptyset\}$. 
    It suffices to show that $E_A$ is c.e.. 
    Now, for each $i\in\N$, we enumerate $i$ into $E_A$ whenever $B(a_i,r_i)\cap A=\emptyset$. 
    Note that $B(a_i,r_i)\cap A=\emptyset$ if and only if $d(a_i,A)>r_i$, which occurs if and only if $d(a_i,C)>r_i+s$. 
    Since $C\in\Pi^0_1(\R)$, $x\mapsto d(x,C)$ is lower-semicomputable (Theorem 5.1.2 in \cite{Weihrauch.2000}). 
    Thus, the enumeration is effective. 
    Since $s>0$ was arbitrary, the result follows.
\end{proof}

\begin{proof}[Proof of Theorem \ref{thm:cpm}]
    Suppose that $\{\mu_n\}_{n\in\N}$ converges effectively in $\rho$ to $\mu$. Then, $\mu$ is computable, and we have a computable function $\epsilon:\N\rightarrow\N$ such that for all $N\in\N$ and all $n\geq\epsilon(N)$, $\rho(\mu_n,\mu)<2^{-N}$. In particular, for any $C\in\Pi^0_1(\R)$ and all $n\geq\epsilon(N)$,
    \[
        \mu_n(C)\leq\mu(B(C,2^{-N}))+2^{-N}\text{ and }\mu(C)\leq\mu_n(B(C,2^{-N}))+2^{-N}.
    \]
    By Theorem \ref{thm:ept}, it suffices to compute an index of a witness that $\limsup_n\mu_n(C)$ is not larger than $\mu(C)$ from an index $e\in\N$ of $C\in\Pi^0_1(\R)$.

    Fix $r\in\Q$. 
    Wait until $r$ is enumerated into the right Dedekind cut of $\mu(C)$. 
    Since $\mu$ is computable and by Lemma \ref{lem:cn}, $\mu(\ol{B(C,2^{-N})})$ is right-c.e. for any $N\in\N$. 
    Search for the first $M_0\in\N$ so that $r-\mu(C)>2^{-M_0}$. 
    Then, search for the first $N_0$ such that $r-\mu(\ol{B(C,2^{-N_0})})>2^{-M_0}$. 
    Let $M=M_0+N_0+1$, and let $n_0=\epsilon(M)$. 
    Therefore, for all $n\geq n_0$,
    \[
        \mu_n(C)\leq\mu(B(C,2^{-M}))+2^{-M}\leq\mu(\ol{B(C,2^{-M})})+2^{-M}<r-2^{-M}+2^{-M}=r.
    \]
    It follows that $n_0$ is an index of a witness that $\limsup_n\mu_n(C)$ is not larger than $\mu(C)$.

    Next, suppose that $\{\mu_n\}_{n\in\N}$ effectively weakly converges to $\mu$. Then, $\mu$ is computable. By Theorem \ref{thm:ept}, we can compute for every $\mu$-almost decidable $A$ an index of a modulus of convergence of $\{\mu_n(A)\}_{n\in\N}$ from a $\mu$-almost decidable index of $A$.

    We build the function $\epsilon:\N\rightarrow\N$ by the following effective procedure. 
    First, let $N\in\N$. 
    By Lemma \ref{lem:eoc}, we can compute a sequence $\{B_j\}_{j\in\N}$ of uniformly $\mu$-almost decidable rational open balls in $\R$ with radius less than $2^{-(N+3)}$ such that $\bigcup_{j=1}^{\infty}B_j=\R$. 
    Search for the first $k_0$ so that $\mu(\bigcup_{j=1}^{k_0}B_j)\geq\mu(\R)-2^{-(N+2)}$. 
    Let
    \[
        \A=\left\{\bigcup_{j\in J}B_j:J\subseteq\{1,\ldots,k_0\}\right\}.
    \]
    Then, $\A$ is a finite collection of $\mu$-almost decidable sets. 
    Define $\epsilon(N)$ to be the smallest index so that $|\mu_n(A)-\mu(A)|<2^{-(N+2)}$ for every $A\in\A$ and every $n\geq\epsilon(N)$. 
    Thus,
    \[
        \mu_n\left(\R\setminus\left(\bigcup_{j=1}^{k_0}B_j\right)\right)\leq\mu\left(\R\setminus\left(\bigcup_{j=1}^{k_0}B_j\right)\right)+2^{-(N+2)}\leq2^{-(N+1)}
    \]
    for all $n\geq\epsilon(N)$.

    To see that $\epsilon$ is the desired function, fix $E\in\B(\R)$. 
    Let
    \[
        A_0=\bigcup\{B_j:(j\in\{1,\ldots,k_0\})\wedge(B_j\cap E\neq\emptyset)\}.
    \]
    Then, $A_0\in\A$ and satisfies the following properties:
    \begin{enumerate}
        \item $A_0\subset B(E,2^{-(N+2)})\subset B(E,2^{-N})$.
        \item $E\subset A_0\cup\left(\R\setminus\left(\bigcup_{j=1}^{k_0}B_j\right)\right)$.
        \item $|\mu_n(A_0)-\mu(A_0)|<2^{-(N+2)}$ for all $n\geq\epsilon(N)$.
    \end{enumerate}
    Therefore, for all $n\geq\epsilon(N)$,
    \begin{align*}
        \mu(E)&\leq\mu(A_0)+2^{-(N+2)}\\
        &<\mu_n(A_0)+2^{-(N+1)}\\
        &\leq\mu_n(B(E,2^{-(N+2)}))+2^{-(N+1)}<\mu_n(B(E,2^{-N}))+2^{-N}
    \end{align*}
    and
    \[
        \mu_n(E)\leq\mu_n(A_0)+2^{-(N+1)}<\mu(A_0)+2^{-N}\leq\mu(B(E,2^{-N}))+2^{-N}.
    \]
    Since $E\in\B(\R)$ was arbitrary, it follows that $\rho(\mu_n,\mu)<2^{-N}$ for all $n\geq\epsilon(N)$, as desired.
\end{proof}

Theorem \ref{thm:cpm} serves as further evidence that effective weak convergence is the appropriate computable analogue to weak convergence. 
When viewing $\M(\R)$ as a computable metric space, effective weak convergence can be defined as the effective topology induced by $\rho$ in $\M(\R)$.

\section{Effective Vague Convergence in $\M(\R)$}

In Sections 3 and 4, we discussed effective weak convergence in $\M(\R)$. 
In this section, we effectivize the definition of vague convergence. 
Just as in the case of weak convergence, we provide a non-uniform definition (Definition \ref{def:ev}) and a uniform definition (Definition \ref{def:uev}).

\begin{definition}\label{def:ev}
    We say $\{\mu_n\}_{n\in\N}$ \emph{effectively vaguely converges} to $\mu$ if for every $f \in C_K^c(\R)$, $\lim_n \int_\R f\ d\mu_n = \int_\R f\ d\mu$ and it is possible to compute an index of a modulus of convergence of $\{\int_\R f\ d\mu_n\}_{n \in \N}$ from an index of $f$ and an index of $\supp{f}$. 
\end{definition}

\begin{definition}\label{def:uev}
    We say $\{\mu_n\}_{n\in\N}$ \emph{uniformly effectively vaguely converges} to $\mu$ if it weakly converges to $\mu$ and there is a uniform procedure that for any $f \in C_K(\R)$ computes a modulus of convergence of $\{\int_\R f\ d\mu_n\}_{n \in \N}$ from a name of $f$ and a name of $\supp{f}$.
\end{definition}

As expected, the following theorem established the equivalence between these two definitions.

\begin{theorem}\label{thm:evc.equiv}
    Suppose $\{\mu_n\}_{n\in\N}$ is uniformly computable. The following are equivalent.
    \begin{enumerate}
        \item $\{\mu_n\}_{n\in\N}$ is effectively vaguely convergent.
        \item $\{\mu_n\}_{n\in\N}$ is uniformly effectively vaguely convergent.
    \end{enumerate}
\end{theorem}

Before we prove this theorem, we need the following lemma.

\begin{lemma}\label{lem:rpf}
    For all $f\in C_K(\R)$, it is possible to compute a rational polygonal function $\psi$ so that $\supp{\psi}\subseteq[\lfloor\min\{\supp{f}\}\rfloor,\lceil\max\{\supp{f}\}\rceil]$ from a name of $f$ and a name of $\supp{f}$.
\end{lemma}

\begin{proof}
    Since $\min\{\supp{f}\}$ and $\max\{\supp{f}\}$ are computable, we may compute rationals $p\in(\lfloor\min\{\supp{f}\}\rfloor,\min\{\supp{f}\})$ and $q\in(\max\{\supp{f}\},\lceil\max\{\supp{f}\}\rceil)$. Now, compute a function $\psi\in P_{\Q}[p,q]$ that approximates $f$ on the interval $[p,q]$ with the property that $\psi(p)=\psi(q)=0$.
\end{proof}

\begin{proof}[Proof of Theorem \ref{thm:evc.equiv}]
    It is possible to compute a name of $f\in C^c_K(\R)$ and a name of $\supp{f}$ from an index of $f$ and an index of $\supp{f}$. It thus follows that every uniformly effectively vaguely convergent sequence is effectively vaguely convergent.

    Now, suppose $\{\mu_n\}_{n\in\N}$ effectively vaguely converges to $\mu$. Let $\rho$ be a name of $f\in C_K(\R)$, and let $\kappa$ be a name of $\supp{f}$. We construct a function $G:\N\rightarrow\N$ as follows. Let $I=[\lfloor\min\{\supp{f}\}\rfloor,\lceil\max\{\supp{f}\}\rceil]$, which can be computed from $\kappa$ by Lemma 5.2.6 in \cite{Weihrauch.2000}. Let $T$ be a function on $\R$ given by
    \[
        T(x)=\begin{cases}
            1&\lceil\min\{\supp{f}\}\rceil\leq x\leq\lceil\max\{\supp{f}\}\rceil\\
            x+\lceil\min\{\supp{f}\}\rceil+1&\lceil\min\{\supp{f}\}\rceil-1<x<\lceil\min\{\supp{f}\}\rceil\\
            -x+\lceil\max\{\supp{f}\}\rceil+1&\lceil\max\{\supp{f}\}\rceil<x<\lceil\max\{\supp{f}\}\rceil+1\\
            0&\text{ otherwise.}
        \end{cases}
    \]
    Note that $\supp{T}=\ol{B(I,1)}$, and so $T\in C_K(\R)$.
    Since $\{\mu_n\}_{n\in\N}$ effectively vaguely converges to $\mu$, we can compute an index $n_0\in\N$ so that $|\int_{\R}Td\mu_n-\int_{\R}Td\mu|<2^{-1}$ whenever $n\geq n_0$. It follows that for every $n,m\geq n_0$, $|\int_{\R}Td\mu_n-\int_{\R}Td\mu_m|<1$.

    By Lemma \ref{lem:rpf}, we can compute a rational polygonal function $\psi$ with the property that $\supp{\psi}\subseteq I$ and
    \[
        \max\{|f(x)-\psi(x)|:x\in I\}<\dfrac{2^{-(N+2)}}{1+\int_{\R}Td\mu_{n_0}}.
    \]
    Since $\{\mu_n\}_{n\in\N}$ effectively vaguely converges to $\mu$, we can compute an $n_1\in\N$ so that $|\int_{\R}\psi d\mu_n-\int_{\R}\psi d\mu|<2^{-(N+1)}$ whenever $n\geq n_1$. Set $G(N)=n_1$. 

    Suppose $n\geq G(N)$. Then,
    \begin{align*}
        \left|\int_{\R}fd\mu_n-\int_{\R}fd\mu\right|&\leq\left|\int_{\R}(f-\psi)d\mu_n\right|+\left|\int_{\R}\psi d\mu_n-\int_{\R}\psi d\mu\right|+\left|\int_{\R}(\psi-f)d\mu\right|\\
        &<2^{-(N+2)}+2^{-(N+1)}+2^{-(N+2)}\\
        &=2^{-N}.
    \end{align*}
    Thus, $G$ is a modulus of convergence of $\{\int_{\R}fd\mu_n\}_{n\in\N}$. Since the construction of $G$ from $\rho$ and $\kappa$ is uniform, $\{\mu_n\}_{n\in\N}$ uniformly effectively vaguely converges to $\mu$.
\end{proof}

By the same reasoning as with effective weak convergence, we may also conclude the following.

\begin{corollary}\label{cor:evc2vc}
    Suppose $\{\mu_n\}_{n\in\N}$ is a uniformly computable sequence in $\M(\R)$ that is effectively vaguely convergent. Then, $\{\mu_n\}_{n\in\N}$ is vaguely convergent.
\end{corollary}

Recall that effective weak limit measures are computable. It is reasonable to ask if this translates to effective vague convergence as well. Below, we provide a negative answer to this question.

\begin{proposition}\label{prop:evil}
    There is a uniformly computable sequence of measures that effectively vaguely converges to a limit measure $\mu$ with the property that $\mu(\R)$ is incomputable. 
\end{proposition}

\begin{proof}
    Let $A\subset\N$ be an incomputable c.e. set, and let $\{a_i\}_{n\in\N}$ be an effective enumeration of $A$. For each $n\in\N$, let $\mu_n=\sum_{i=0}^n2^{-(a_i+1)}\delta_i$. Note that $\mu_n$ is a uniformly computable sequence of measures. We will show that $\{\mu_n\}_{n\in\N}$ effectively vaguely converges to the measure $\mu=\sum_{i=0}^{\infty}2^{-(a_i+1)}\delta_i$.

    For starters, fix $f\in C^c_K(\R)$ and an index of $\supp{f}$. Then, we can compute $\max\{\supp{f}\}$ by Lemma 5.2.6 in \cite{Weihrauch.2000}. Observe that for all $n\geq\lceil\max\{\supp{f}\}\rceil+1$,
    \[
        \int_{\R}fd\mu_n=\sum_{i=0}^n2^{-(a_i+1)}f(i)=\sum_{i=0}^{\infty}2^{-(a_i+1)}f(i)=\int_{\R}fd\mu.
    \]
    Thus, $\max\{\supp{f}\}\rceil+1$ is an index of a modulus of convergence of the sequence $\{\int_{\R}fd\mu_n\}_{n\in\N}$. Finally, note that $\mu(\R)=\sum_{i=0}^{\infty}2^{-(a_i+1)}$ is incomputable since it is the limit of a Specker sequence.
\end{proof}

We have in the proof of Proposition \ref{prop:evil} above not only an example of an incomputable effective vague limit, but also a \emph{finite} one.
This leads us to ask the following question: when are effective vague limit measures computable? We provide a necessary and sufficient condition for which this is the case.

\begin{proposition}
    Suppose $\{\mu_n\}_{n\in\N}$ is a uniformly computable sequence in $\M(\R)$ that effectively vaguely converges to $\mu$. If $\mu(\R)$ is computable, then $\mu$ is computable.
\end{proposition}

\begin{proof}
    Suppose $\mu(\R)$ is computable. It suffices to show that $\mu(U)$ is left-c.e. for every $U\in\Sigma^0_1(\R)$ uniformly in an index of $U$.

    For starters, let $I$ be a rational open interval. Since $\one_{I}$ is nonnegative and lower-semicomputable, we can compute a sequence of computable Lipschitz functions $0\leq T_k\leq1$ such that $T_k$ increases to $\one_{I}$ pointwise and $\supp{T_k}=\ol{I}$ for each $k$ (see {Proposition C.7, \cite{G05}}). By the Monotone Convergence Theorem, $\nu(I)=\lim_k\int_{\R}T_kd\nu$ for any $\nu\in\M(\R)$.

    Fix $q\in\Q$. Since $\{\mu_n\}_{n\in\N}$ effectively vaguely converges to $\mu$, we can compute an index of a modulus of convergence of $\{\int_{\R}T_kd\mu_N\}_{n\in\N}$ with $\lim_n\int_{\R}T_kd\mu_n=\int_{\R}T_kd\mu$ for each $k$ from an index of $T_k$ and an index of $\supp{T_k}$. This means $\int_{\R}T_kd\mu$ is a computable real for each $k$. Thus, we enumerate $q$ into the left Dedekind cut of $\mu(I)$ if we can find $k_1,k_2\in\N$ such that $k_1<k_2$ and $\int_{\R}T_{k_1}d\mu\leq q\leq\int_{\R}T_{k_2}d\mu$. It follows that $\mu(I)$ is left-c.e. uniformly in an index of $I$.

    Now, fix $U\in\Sigma^0_1$. Then, $U$ can be expressed as a countable union of rational open intervals. By the observation above, it follows that $\mu(U)$ is the limit of an increasing sequence of left-c.e. reals. Therefore, $\mu(U)$ is left-c.e. uniformly in an index of $U$.
\end{proof}

Another way of ensuring that effective vague limits are computable is by analyzing the point in which effective weak and vague convergence coincide. Below, we provide a sufficient condition for which these notions do coincide.

\begin{theorem}\label{thm:evc2ewc}
    Suppose $\{\mu_n\}_{n\in\N}$ is uniformly computable. Suppose further that there is a computable modulus of convergence of $\{\mu_n(\R)\}_{n\in\N}$. The following are equivalent.
    \begin{enumerate}
        \item $\{\mu_n\}_{n\in\N}$ is effectively vaguely convergent.
        \item $\{\mu_n\}_{n\in\N}$ is effectively weakly convergent.
    \end{enumerate}
\end{theorem}

In fact, we will prove that effective vague convergence is equivalent to \emph{uniform} effective weak convergence. The following series of lemmas will allow us to carry out a proof of Theorem \ref{thm:evc2ewc} similar to the proof of Theorem \ref{thm:ewc.equiv}.

\begin{lemma}\label{lem:1-f}
    Suppose $\{\mu_n\}_{n\in\N}$ is uniformly computable and effectively vaguely converges to $\mu$. 
    Suppose further that there is a computable modulus of convergence of $\{\mu_n(\R)\}_{n\in\N}$. 
    For every $f\in C^c_K(\R)$ such that $0\leq f\leq 1$, $\lim_n\int_{\R}(1-f)d\mu_n=\int_{\R}(1-f)d\mu$ and it is possible to compute an index of a modulus of convergence of $\{\int_{\R}(1-f)d\mu_n\}_{n\in\N}$ from an index of $f$ and an index of $\supp{f}$.
\end{lemma}

\begin{proof}
    Fix $f\in C^c_K(\R)$ and an index of $\supp{f}$. By Theorem \ref{thm:evc.equiv}, we can compute a modulus of convergence $g_1:\subseteq\N\rightarrow\N$ for $\{\int_{\R}fd\mu_n\}_{n\in\N}$. By assumption, there is a computable modulus of convergence $g_2:\subseteq\N\rightarrow\N$ for $\{\mu_n(\R)\}_{n\in\N}$. Therefore, $n_1=\max\{g_1(N+1),g_2(N+1)\}$ is an index of a modulus of convergence of $\{\int_{\R}(1-f)d\mu_n\}_{n\in\N}$.
\end{proof}

\begin{lemma}\label{lem:sc1}
    Suppose $\{\mu_n\}_{n\in\N}$ is uniformly computable and effectively vaguely converges to $\mu$. 
    Suppose further that there is a computable modulus of convergence of $\{\mu_n(\R)\}_{n\in\N}$. 
    From $N\in\N$, it is possible to compute $a,n_0\in\N$ such that $\mu_n(\R\setminus[-a,a])<2^{-N}$ for all $n\geq n_0$.
\end{lemma}

\begin{proof}
    Combine Lemma 2.5 and the proof of Lemma 4.5 in \cite{MR21}.
\end{proof}

\begin{lemma}\label{lem:sc2}
    Suppose $\{\mu_n\}_{n\in\N}$ is uniformly computable and effectively vaguely converges to $\mu$. 
    Suppose further that there is a computable modulus of convergence of $\{\mu_n(\R)\}_{n\in\N}$. 
    From a name of an $f\in C_b(\R)$ and an $N,B\in\N$ so that $|f|\leq B$, it is possible to compute $a,n_1\in\N$ and $\psi\in P_{\Q}[-a,a]$ so that $\psi$ is computably compactly supported, $|\int_{\R}(f-\psi)d\mu|<2^{-N}$, and $|\int_{\R}(f-\psi)d\mu_n|<2^{-N}$ whenever $n\geq n_1$.
\end{lemma}

\begin{proof}
    We modify the proof of Lemma 4.6 in \cite{MR21} so that $\psi$ is computably compactly supported. Then, use a name of $\psi$ and a name of $\supp{\psi}$ to compute $n_1\in\N$ with the desired properties.
\end{proof}

\begin{proof}[Proof of Theorem \ref{thm:evc2ewc}]
    It is immediate that $(2)$ implies $(1)$. Now, suppose $\{\mu_n\}_{n\in\N}$ effectively vaguely converges to $\mu$. Fix $f\in C_b(\R)$ with name $\rho$ and bound $B\in\N$. We construct the function $G:\N\rightarrow\N$ as follows. By means of Lemma \ref{lem:sc2}, we can compute $a,n_1\in\N$ and $\psi\in P_{\Q}[-a,a]$ so that $\psi$ is computably compactly supported, $|\int_{\R}(f-\psi)d\mu|<2^{-N}$, and $|\int_{\R}(f-\psi)d\mu_n|<2^{-N}$ whenever $n\geq n_1$. Since $\{\mu_n\}_{n\in\N}$ effectively vaguely converges to $\mu$, we can compute an $n_2\in\N$ so that $|\int_{\R}\psi d\mu_n-\int_{\R}\psi d\mu|<2^{-(N+1)}$ whenever $n\geq n_2$. Set $G(N)=n_2$.

    Suppose $n\geq G(N)$. Then,
    \begin{align*}
        \left|\int_{\R}fd\mu_n-\int_{\R}fd\mu\right|&\leq\left|\int_{\R}(f-\psi)d\mu_n\right|+\left|\int_{\R}\psi d\mu_n-\int_{\R}\psi d\mu\right|+\left|\int_{\R}(\psi-f)d\mu\right|\\
        &<2^{-(N+2)}+2^{-(N+1)}+2^{-(N+2)}\\
        &=2^{-N}.
    \end{align*}
    Thus, $G$ is a modulus of convergence of $\{\int_{\R}fd\mu_n\}_{n\in\N}$. Since the construction of $G$ from $\rho$ and $B$ is uniform, $\{\mu_n\}_{n\in\N}$ uniformly effectively weakly converges to $\mu$. The result follows by Theorem \ref{thm:ewc.equiv}.
\end{proof}

We conclude this section by deriving from Theorem \ref{thm:evc2ewc} the following effective version of a classical result in probability theory.

\begin{corollary}\label{cor:evc2ewc.prob}
    Suppose $\{\mu_n\}_{n\in\N}$ is a uniformly computable sequence of probability measures. The following are equivalent.
    \begin{enumerate}
        \item $\{\mu_n\}_{n\in\N}$ is effectively vaguely convergent.
        \item $\{\mu_n\}_{n\in\N}$ is effectively weakly convergent.
    \end{enumerate}
\end{corollary}

\section{Conclusion}

We expanded the effective framework for the study of weak convergence of measures in $\M(\R)$ introduced in \cite{MR21} by demonstrating the equivalence between effective weak convergence and effective convergence in the Prokhorov metric. 
This provides further evidence that effective weak convergence is the appropriate analogue to classical weak convergence in $\M(\R)$. 
While the Prokhorov metric is useful in defining $\M(\R)$ as a computable metric space, effective weak convergence is a more useful tool to analyze properties of $\M(\R)$ as a computable metric space. 
Theorem \ref{thm:cpm}, therefore, unifies the approaches in \cite{HR09b} and \cite{MR21} to studying the effective theory of weak convergence in $\M(\R)$.

Additionally, we introduced two effective notions of vague convergence in $\M(\R)$. 
While the moduli of convergence in the first definition are produced for computable functions in $C_K(\R)$, moduli of convergence in the second definition are produced for all functions in $C_K(\R)$ via names. 
Similar to effective weak convergence, Theorem \ref{thm:evc.equiv} shows that they are equivalent. 
Just as in the classical sense, however, there are notable differences between effective weak convergence and effective vague convergence.

Consider the following classical example. 
The sequence $\{\delta_n\}_{n\in\N}$ of point masses converges vaguely to the zero measure, but it does not converge weakly since $\lim_nf(n)=\infty$ for any $f\in C_b(\R)$ supported on $\R$. 
This distinction carries over in the effective setting, albeit in a more computationally significant manner. 
In Proposition \ref{prop:evil}, we gave an example of a uniformly computable sequence of measures that effectively vaguely converges to a finite incomputable measure. 
Thus, the ``vagueness'' of effective vague convergence is present in the fact that limits under this convergence notion may not be computable even when finite.

Nevertheless, we provide evidence that effective vague convergence is the appropriate computable analogue to classical vague convergence. 
For instance, we found in Theorem \ref{thm:evc2ewc} a sufficient condition for which effective weak and vague convergence coincide.
Consequently, Corollary \ref{cor:evc2ewc.prob} provides the following observation: whereas classical weak and vague convergence coincide at the probability measures, effective weak and vague convergence coincide at the computable probability measures. 
Since we argue that effective weak convergence is the appropriate computable analogue to classical weak convergence, a similar argument follows in the case of effective vague convergence. 
In the future, we would like to generalize the definitions of effective weak and vague convergence to measures in $\M(X)$ for an arbitrary computable metric space $X$.

\section{Acknowledgements}

We would like to thank Timothy McNicholl for proofreading and providing several valuable comments and suggestions.

\section{Declarations}

The author did not receive support from any organization for the submitted work. The author has no financial or proprietary interests in any material discussed in this article. Data sharing is not applicable to this article as no datasets were generated or analyzed during the current study.

\bibliographystyle{plain}
\bibliography{drbib}
    
\end{document}